\documentclass[a4paper,11pt]{article}
\usepackage{amssymb}
\usepackage{amsfonts}
\usepackage[english]{babel}
\usepackage{amsmath,amscd}
\usepackage{amsthm}
\usepackage{a4wide}
\title{Endomorphisms of quantum generalized Weyl algebras}
\begin{document}
\author{A.P. Kitchin\thanks{APK thanks EPSRC for its support.}  ~and S. Launois\thanks{SL is grateful for the
financial support of EPSRC first grant \textit{EP/I018549/1}.}}
\date{ }

\maketitle

\theoremstyle{theorem}
\newtheorem{theorem}{Theorem}[section]
\newtheorem{definition}[theorem]{Definition}
\newtheorem{lemma}[theorem]{Lemma}
\newtheorem{example}[theorem]{Example}
\newtheorem{conjecture}[theorem]{Conjecture}
\newtheorem{proposition}[theorem]{Proposition}
\newtheorem{corollary}[theorem]{Corollary}
\newtheorem{remark}[theorem]{Remark}
\begin{abstract}
We prove that every endomorphism of a simple quantum generalized Weyl algebra $A$ over a commutative Laurent polynomial ring in one variable is an automorphism. This is achieved by obtaining an explicit classification of all endomorphisms of $A$. Our main result applies to minimal primitive factors of the quantized enveloping algebra of $U_q(\mathfrak{sl}_2)$ and certain minimal primitive quotients of the positive part of $U_q(\mathfrak{so}_5)$. 
\end{abstract}

\vskip .5cm
\noindent
{\em 2010 Mathematics subject classification:} 16W35, 16S32, 16W20, 17B37

\vskip .5cm
\noindent
{\em Key words:} Generalized Weyl Algebras; Endomorphisms

\section{Introduction.}
At the end of his foundational paper on the Weyl algebra, \cite{Dixmier}, Dixmier posed the now famous conjecture that any algebra endomorphism of the Weyl algebras is an automorphism. Tsuchimoto, \cite{Tsuchimoto}, and Belov-Kanel and Kontsevich, \cite{Belov}, proved independently that the Dixmier Conjecture is stably equivalent to the Jacobian Conjecture of Keller \cite{Keller}. To the best of the authors' knowledge these conjectures remain unproven. The main aim of this paper is to exhibit algebras related to the first Weyl algebra, and which possess the property that every endomorphism is an automorphism.

Let $\mathfrak{n}$ be a finite dimensional complex nilpotent Lie algebra. The structure of the primitive factor algebras of the enveloping algebra $U(\mathfrak{n})$ of $\mathfrak{n}$ are well known: a theorem of Dixmier asserts that these factor algebras are actually isomorphic to Weyl algebras (see \cite{Dixmier2}). Let $\mathbb{K}$ be a field and $q \in \mathbb{K}^*$ not a root of unity. In \cite{Launois},  primitive factor algebras of Gelfand-Kirillov dimension 2 of the positive part $U_q^+(\mathfrak{so}_5)$ of the quantized enveloping algebra $U_q(\mathfrak{so}_5)$ were classified, and can be thought of as quantum analogues of the first Weyl algebra. Among those are the algebras $A_{\alpha,q}$ with $\alpha \neq 0$, where $A_{\alpha,q}$ is the $\mathbb{K}$-algebra generated by $e_1, e_2$ and $e_3$ subject to the commutation relations:
\begin{align*}
& e_1 e_3=q^{-2}e_3 e_1,\\
& e_2 e_3=q^2 e_3 e_2+\alpha,\\
& e_2e_1=q^{-2} e_1 e_2 -q^{-2}e_3,\\
& e_2^2+(q^4-1) e_3 e_1 e_2 +\alpha q^2(q^2+1) e_1=0.
\end{align*}
If $\alpha\neq 0$ then $A_{\alpha,q}$ is simple (see \cite{Launois}). Note that setting $q=1$ and $\alpha=1$ we indeed get an algebra isomorphic to the first Weyl algebra $A_1(\mathbb{K})$. In this paper, we prove the following theorem which can be thought of as an analogue of the Dixmier Conjecture for this class of algebras.
\begin{theorem}\label{aqthe}
Every endomorphism of $A_{\alpha,q}$ is an automorphism when $q$ is not a root of unity and $\alpha$ nonzero.
\end{theorem}
\noindent The algebras $A_{\alpha,q}$ ($\alpha \neq 0$) are specific examples of so-called quantum generalized Weyl algebras (over a Laurent polynomial ring in one variable) in the sense of \cite{Bavula1}, and so we will prove a more general result for this class of algebras. More precisely, let $\mathbb{K}$ be a field with $q\in \mathbb{K}^*$ not a root of unity, and let $a(h) \in \mathbb{K}[h^{\pm 1}]$ be a Laurent polynomial. Recall that the quantum generalized Weyl algebra $A(a(h),q)$ is the $\mathbb{K}$-algebra generated by $x$ and $y$ subject to the relations
\begin{equation}\label{qGWArel}
xh=qhx, ~~ yh=q^{-1}hy,~~ xy=a(qh), ~~ yx=a(h), ~~h^{\pm 1}h^{\mp 1}=1.
\end{equation} 
The algebra $A(a(h),q)$ is $\mathbb{Z}$-graded \cite{Bavula2} with $\deg(h)=0$, $\deg(x)=1$ and $\deg(y)=-1$, and it is well known that  
\begin{equation}\label{grading}
A(a(h),q)=\displaystyle\bigoplus_{n\in \mathbb{Z}}A_n, 
\end{equation}
where $A_n=\mathbb{K}[h^{\pm1}]x^n$ if $n\geq 0$ and $A_n=\mathbb{K}[h^{\pm1}]y^{-n}$ if $n<0$.

Examples of  quantum generalized Weyl algebra include the quantum torus in 2 indeterminates, the minimal primitive factors of the quantized enveloping algebra of $U_q(\mathfrak{sl}_2)$ (see \cite[Example 5.3]{Bavula1}) and the algebras $A_{\alpha,q}\simeq A(\frac{h^{-1}}{1-q^4}-\frac{\alpha}{q^2-1},q^2)$ (see \cite[Proposition 3.10]{Launois}). Quantum generalized Weyl algebras are special examples of generalized Weyl algebras which were introduced by Bavula in \cite{Bavula2} and have been widely studied. The problem of when two quantum generalized Weyl algebras are isomorphic was solved by Bavula and Jordan in \cite{Bavula1} for $q$ not a root of unity. The same problem when working over a polynomial ring was addressed by Richard and Solotar in \cite{RichardandSolotar}.

In this paper we classify all endomorphisms of quantum generalized Weyl algebras and obtain the following result which is a consequnce of Section \ref{section2}, Corollaries \ref{endaut+} and \ref{endaut0} and Proposition \ref{Propneg2}.

\begin{theorem}\label{aqthe2}
Every monomorphism of a quantum generalized Weyl algebra is an automorphism. 
\end{theorem}
As a consequence, every endomorphism of a simple quantum generalized Weyl algebra is an automorphism, so that Theorem \ref{aqthe}  is just a consequence of Theorem \ref{aqthe2}. Previously, the fact that  every endomorphism of a quantum generalized Weyl algebra $A(a(h),q)$ is an automorphism, \cite{Richard}, has only been established, by Richard, in the case where $a(h)$ is a (Laurent) monomial.  Indeed, in the latter case,  $A(a(h),q)$ is a quantum torus in two variables, and Richard proved that every endomorphism of a simple McConnell-Pettit algebra is an automorphism. Backelin, \cite{Backelin}, proved that quantum Weyl algebras have monomorphisms which are not automorphisms, therefore our result cannot be extended to quantum generalized Weyl algebras over a polynomial ring in one variable. 

Our strategy to classify the endomorphisms of $A(a(h),q)$ is to first study the action of these endomorphisms on the group of units of $A(a(h),q)$, and then to take advantage of the graduation of $A(a(h),q)$.

After this research was completed, we learned about the second version of \cite{Vivas}, where the automorphism groups of quantum generalized Weyl algebras are described. As a consequence of our classification of all endomorphisms, we also retrieve these automorphism groups. Our results show that the description obtained in \cite{Vivas} was incomplete.

\section{Endomorphisms of $A(a(h),q)$: Trichotomy}\label{section2}

Throughout this paper we assume that $q \in \mathbb{K}^*$ is not a root of unity and we fix $a(h)=\displaystyle\sum_{i=e}^d a_i h^i\in\mathbb{K}[h^{\pm1}]$, where  $d$ and $e$ are the degree and valuation of $a(h)$ respectively.  In the following sections,  we will need to know which coefficients of $a(h)$ are zero, so we also introduce the following notation:
\begin{displaymath} 
a(h):=\displaystyle\sum_{j=1}^m a_{i_j} h^{i_j},
\end{displaymath}
where $e=i_1< i_2< \ldots < i_m=d \mbox{ and } a_{i_j} \neq 0 \mbox{ for all }j$.

Finally we set 
\begin{eqnarray}
t:=\gcd(d-i_1,d-i_2,\ldots,d-i_m).
\end{eqnarray}

In the case where $a(h)$ is invertible in $\mathbb{K}[h^{\pm 1}]$, the algebra $A(a(h),q)$ is isomorphic to a quantum torus in two indeterminates. More precisely, when $a(h)$ is invertible in $\mathbb{K}[h^{\pm 1}]$, it is easy to check that 
$A(a(h),q)$ is the $\mathbb{K}$-algebra generated by $h^{\pm1}$ and $x^{\pm 1}$ subject to the relations $h^{\pm 1}h^{\mp 1}=1$, $x^{\pm 1}x^{\mp 1}=1$ and $xh=qhx$. In this case, it was proved by Richard \cite{Richard} that every endomorphism is an automorphism. We henceforth assume that {\bf $a(h)$ is not invertible} in this article. 

One barrier in classifying endomorphisms for the first Weyl algebra $A_1(\mathbb{C})$ is that the group of units is trivial. The algebra $A(a(h),q)$ has a non-trivial group of units. More precisely, it follows from \cite[Lemma 5.1]{Bavula1} that the group of units of $A(a(h),q)$ is equal to the group of units of $\mathbb{K}[h^{\pm1}]$. This greatly restricts the shape of a possible endomorphism of $A(a(h),q)$.

Indeed, let $\psi$ be an algebra endomorphism of $A(a(h),q)$. Since endomorphisms preserve units we have $\psi(h)=\alpha h^i$ where $\alpha\in \mathbb{K}^*$ and $i\in \mathbb{Z}$. The action of $\psi$ on $h$ allows us to study a trichotomy of endomorphisms determined by the exponent of $h$ under $\psi$. More precisely, in the next sections, we will consider three disjoint families of endomorphisms: {\it positive type endomorphisms}, that is, endomorphisms $\psi$ for which $\psi(h)=\alpha h^j$ where $\alpha \in \mathbb{K}^*$ and $j\in \mathbb{Z}_{>0}$; {\it negative type endomorphisms},  that is, endomorphisms $\psi$ for which $\psi(h)=\alpha h^j$ where $\alpha \in \mathbb{K}^*$ and $j\in \mathbb{Z}_{<0}$; and {\it zero type endomorphisms},  that is, endomorphisms $\psi$ for which $\psi(h)=\alpha$, where $\alpha \in \mathbb{K}^*$.

\section{Positive type endomorphisms}\label{secpos}

We keep the notation of the previous section. In particular, $a(h)=\displaystyle\sum_{i=e}^d a_i h^i = \sum_{j=1}^m a_{i_j} h^{i_j}$, where $e=i_1< i_2< \ldots < i_m=d \mbox{ and } a_{i_j} \neq 0 \mbox{ for all }j$, and $t=\gcd(d-i_1,d-i_2,\ldots,d-i_m)$.\\ 

\begin{proposition}\label{propclass1} \
\begin{enumerate}
\item Let $\psi$ be a positive type endomorphism of $A(a(h),q)$. Then $\psi$ acts on the generators of $A(a(h),q)$ as follows:
\begin{equation*}
\psi(h)=\alpha h, ~~\psi(x)=b h^n x, ~~ \psi(y)= y h^{-n} \alpha^d b^{-1},
\end{equation*}
where $\alpha$ is a $t^{\mathrm{th}}$ root of unity, $b\in\mathbb{K}^*$ and $n\in \mathbb{Z}$. \\
\item Conversely if $\alpha$ is a $t^{\mathrm{th}}$ root of unity, $b \in \mathbb{K}^*$ and $n\in \mathbb{Z}$, then there is a unique (positive type) endomorphism $\psi_{(\alpha,n,b)}$ of $A(a(h),q)$ defined on the generators of $A(a(h),q)$ as follows:
\begin{equation*}
\psi_{(\alpha,n,b)}(h)=\alpha h, ~~\psi_{(\alpha,n,b)}(x)=b h^n x, ~~ \psi_{(\alpha,n,b)}(y)= y h^{-n} \alpha^d b^{-1}.
\end{equation*}
\end{enumerate}
\end{proposition}
\begin{proof}
\begin{enumerate}
\item By assumption $\psi(h)=\alpha h^j$ for $\alpha \in \mathbb{K}^*$ and $j\in \mathbb{Z}_{>0}$. Applying $\psi$ to the relation $xh=qhx$ we get
\begin{equation}\label{eqpq12}
\psi(x)\alpha h^j=q\alpha h^j \psi(x).
\end{equation}
Using the $\mathbb{Z}$-graded structure of $A(a(h),q)$ (see Equation (\ref{grading})), we can write 
$$\psi(x)=\sum_{k\in \mathbb{Z}} w_k,$$
where $w_k\in A_k$ is the component of degree $k$ of $\psi(x)$. Equation (\ref{eqpq12}) can now be rewritten as \begin{equation*}
h^{-j}\displaystyle\sum_{k\in \mathbb{Z}} w_k h^j=q \displaystyle\sum_{k\in \mathbb{Z}} w_k.
\end{equation*}
This implies that $q^{jk} w_k=q w_k$ for all $k\in \mathbb{Z}$. By applying $\psi$ to the relation $yx=a$ and comparing degrees in $h$ we see that $\psi(x)$ is non-zero. Therefore there is at least one $k$ such that $w_k\neq 0$ and since $q$ is not a root of unity and $j>0$, the equality $q^{jk}=q$ implies that $j=1$ and $k=1$ (for all $k$ such that $w_k \neq 0$). 
Hence, we have just proved that 
\begin{equation*}
\psi(h)=\alpha h  \mbox{ and }\psi(x)=b(h) x, 
\end{equation*}
where $b(h) \in \mathbb{K}[h^{\pm 1}]$.
Repeating this argument on the relation $yh=q^{-1}hy$ gives us the following action of $\psi$ on $y$:
\begin{equation*}
 \psi(y)= y c(h),
\end{equation*}
where $c(h) \in \mathbb{K}[h^{\pm 1}]$. By applying $\psi$ to $xy=a(qh)$ and comparing degrees in $h$ we find that $b(h)$ and $c(h)$ have degrees and valuations which sum to zero and thus, are Laurent monomials of opposite degrees, say $b(h)=bh^n$ and $c(h)=ch^{-n}$ with $b,c\in \mathbb{K}^*$ and $n \in \mathbb{Z}$. Thus the equation $\psi(x)\psi(y)=\psi(a(qh))$ becomes $b c a(qh)= a(\alpha q h)$. By substituting $a=\displaystyle\sum_{j=1}^m a_{i_j} h^{i_j}$ we get
\begin{equation}\label{eqcoef}
b c \displaystyle\sum_{j=1}^m a_{i_j} q^{i_j} h^{i_j}=\displaystyle\sum_{j=1}^m \alpha^{i_j} a_{i_j} q^{i_j} h^{i_j}.
\end{equation}
Comparing coefficients in Equation $(\ref{eqcoef})$ we see that
\begin{align}\label{idcoef}
& b c a_{i_j} q^{i_j}= \alpha^{i_j} a_{i_j} q^{i_j}
\end{align}
for all $j\in \{ 1, \dots ,m\}$. Equation $(\ref{idcoef})$ implies that $bc=\alpha^{i_j}$ for all ${i_j}$. Since by assumption $a(h)$ is not a monomial there is at least one ${i_j}\neq d=i_m$, and $\alpha^{d-i_j}=1$ for all ${i_j}\neq d$. This implies $\alpha$ is a $t^\text{th}$ root of unity and $bc=\alpha^d$, as desired.\\
\item It suffices to show that $\psi_{(\alpha,n,b)}$ is consistent on the defining relations of $A(a(h),q)$. Set $c:=\alpha^d b^{-1}$, so that $\psi_{(\alpha,n,b)}(y)=yh^{-n}c$. The first two relations follow by the standard calculations
\begin{equation*}
\psi_{(\alpha,n,b)}(x)\psi_{(\alpha,n,b)}(h)=bh^nx\alpha h=q\alpha h b h^n x=q\psi_{(\alpha,n,b)}(h)\psi_{(\alpha,n,b)}(x)
\end{equation*}
and
\begin{equation*}
\psi_{(\alpha,n,b)}(y)\psi_{(\alpha,n,b)}(h)=yh^{-c}\alpha h=q^{-1} \alpha h y h^{-n} c=q^{-1}\psi_{(\alpha,n,b)}(h)\psi_{(\alpha,n,b)}(y).
\end{equation*}
The remaining two relations follow from the assumptions that $\alpha$ is a $t^{\mathrm{th}}$ root of unity and that $bc=\alpha^d$. We see that

\begin{align*}
\psi_{(\alpha,n,b)}(y)\psi_{(\alpha,n,b)}(x) &=bc a(h) = \alpha^d \displaystyle\sum_{j=1}^m a_{i_j} h^{i_j}\\
&=\alpha^d\displaystyle\sum_{j=1}^m \alpha^{i_j-d} a_{i_j} h^{i_j}=a(\alpha h)=\psi_{(\alpha,n,b)}(a)
\end{align*}

and
\begin{align*}
\psi_{(\alpha,n,b)}(x)\psi_{(\alpha,n,b)}(y)&=bc a(qh)=\alpha^d\displaystyle\sum_{j=1}^m q^{i_j} a_{i_j} h^{i_j}\\
&=\alpha^d\displaystyle\sum_{j=1}^m \alpha^{i_j-d} q^{i_j} a_{i_j} h^{i_j}=a( \alpha qh)=\psi_{(\alpha,n,b)}(a(qh)).
\end{align*}
Thus, $\psi_{(\alpha,n,b)}$ is well-defined and is an endomorphism of $A(a(h),q)$.
\end{enumerate}
\end{proof}
\noindent
We denote by $\mathrm{U}_t\subseteq \mathbb{K}^*$ the group of $t^{\mathrm{th}}$ roots of unity. Proposition \ref{propclass1} shows that there is a bijection between the set 
\begin{displaymath}
I:=\{(\alpha,n,b)| \alpha \in U_{t}, n\in \mathbb{Z}, b\in \mathbb{K}^*\}
\end{displaymath} 
and the set $\mathrm{End}_+$ of positive type endomorphisms. This bijection sends $(\alpha,n,b)\in I$ to $\psi_{(\alpha,n,b)}$. Note that the composition of two endomorphisms of positive type is again an endomorphism of positive type. More precisely, it is easy to check that 
\begin{equation}\label{composition}
\psi_{(\alpha_1,n_1,b_1)} \circ \psi_{(\alpha_2,n_2,b_2)}=\psi_{(\alpha_1\alpha_2,n_1+n_2,b_1 b_2 \alpha_1^{n_2})}.
\end{equation}
Observing that $(1,0,1)\in I$ and $\psi_{(1,0,1)}=\mathrm{id}$, we obtain that every positive endomorphism is an automorphism. We have just proved the following result. 
\begin{corollary}\label{endaut+}
Any endomorphism $\psi_{(\alpha,n,b)}$ of positive type is an automorphism with 
$$\psi_{(\alpha,n,b)}^{-1}=\psi_{(\alpha^{-1},-n,b^{-1}\alpha^n)}.$$
\end{corollary}
This result together with (\ref{composition}) shows that the set $\mathrm{End}_+$ is a subgroup of $\mathrm{Aut}(A)$. Moreover, endomorphisms of the form $\psi_{(1,n,1)}$, with $n \in \mathbb{Z}$, form a subgroup of $\mathrm{End}_+$ which is isomorphic to $\mathbb{Z}$; and endomorphisms of the form $\psi_{(\alpha,0,b)}$, with $(\alpha,b) \in \mathrm{U}_t \times \mathbb{K}^*$, form a normal  subgroup of $\mathrm{End}_+$ which is isomorphic to $\mathrm{U}_t \times \mathbb{K}^*$. This allows us to describe the group structure of $\mathrm{End}_+$ in the following summarizing proposition.
\begin{proposition}
The set of positive type endomorphisms of $A(a(h),q)$ is a subgroup of $\mathrm{Aut}(A)$, denoted $\mathrm{End}_+$. Moreover, there is a right split short exact sequence of groups
\begin{equation*}
1\longrightarrow\mathrm{U}_t \times \mathbb{K}^*\longrightarrow \mathrm{End}_+ \longrightarrow \mathbb{Z} \longrightarrow 1.
\end{equation*}
Both $\mathrm{U}_t \times \mathbb{K}^*$ and $\mathbb{Z}$ can be realized as subgroups of $\mathrm{End}_+$. Thus 
\begin{displaymath}
\mathrm{End}_+\cong(\mathrm{U}_t \times \mathbb{K}^*) \rtimes\mathbb{Z}.
\end{displaymath}
\end{proposition}

\section{Zero type endomorphisms}
\begin{proposition}\label{propex}
\begin{enumerate}
\item Let $\psi$ be an endomorphism of zero type, so that $\psi(h)=\alpha$ with $\alpha \in \mathbb{K}^*$. Then $\alpha$ and $q \alpha$ are roots of $a(h)$, and $\psi(x)=0=\psi(y)$.
\item Conversely, if $\alpha$ and $q \alpha$ are roots of $a(h)$, then there exists a unique (zero type) endomorphism $\psi_{\alpha}$ of $A(a(h),q)$ such that 
$$\psi_{\alpha}(h)=\alpha \mbox{ and } \psi_{\alpha}(x)=0=\psi_{\alpha}(y).$$
\end{enumerate}
\end{proposition}
\begin{proof}
\begin{enumerate}
\item Assume first that  $\psi$ is an endomorphism of $A(a(h),q)$. Applying $\psi$ to the relation $xh=qhx$ we get that $\alpha(1-q) \psi(x) =0$. As $\alpha\neq 0$ and $q$ is not a root of unity this implies $\psi(x)=0$. Similarly we find that $\psi(y)=0$. Applying $\psi$ to the relations $yx=a$ and $xy=a(qh)$ we see $\alpha$ and $q\alpha$ are roots of $a(h)$.\\
\item Conversely suppose that $\alpha$ and $q\alpha$ are roots of $a(h)$. It suffices to check that $\psi_{\alpha}$ is consistent on the defining relations of $A(a(h),q)$. We see by standard calculation that
\begin{align*}
&\psi_{\alpha}(x)\psi_{\alpha}(h)=0=q\psi_{\alpha}(h)\psi_{\alpha}(x),\\
&\psi_{\alpha}(y)\psi_{\alpha}(h)=0=q^{-1}\psi_{\alpha}(h)\psi_{\alpha}(y).
\end{align*}
Since $\alpha$ and $q \alpha$ are roots of $a(h)$ we get
\begin{equation*}
\psi_{\alpha}(x)\psi_{\alpha}(y)=0=\psi_{\alpha}(a)
\end{equation*}
and
\begin{equation*}
\psi_{\alpha}(y)\psi_{\alpha}(x)=0=\psi_{\alpha}(a(qh)),
\end{equation*}
as desired.
\end{enumerate}
\end{proof}
\noindent In the case where $\alpha$ and $q \alpha$ are both roots of $a(h)$, the algebra  $A(a(h),q)$ is not simple by \cite[Proposition 3.1.2]{RichardandSolotar}. Thus, the following result follows immediately from Proposition \ref{propex}.

\begin{corollary}\label{endaut0}
\begin{enumerate}
\item $A(a(h),q)$ has no zero type monomorphisms. 
\item If $A(a(h),q)$ is simple, then it has no zero type endomorphisms.
\end{enumerate}
\end{corollary}

\section{Negative type endomorphisms}\label{section:negative}

Recall that an endomorphism $\psi$ of $A(a(h),q)$ has negative type if there exist $\alpha \in \mathbb{K}^*$ and a positive integer $j$ such that $\psi(h)=\alpha h^{-j}$. 

\begin{proposition}\label{Propneg2}
Every endomorphism of negative type is an automorphism.
\end{proposition}

\begin{proof}
Let $\psi$ be a negative type endomorphism with $\psi(h)=\alpha h^{-j}$, $\alpha \in \mathbb{K}^*$ and $j\in \mathbb{Z}_{>0}$. We see that $\psi^2(h)=\alpha^{1-j}h^{j^2}$, so that $\psi^2$ is a positive type endomorphism and is therefore an automorphism. We conclude that $\psi$ is an automorphism.
\end{proof}
\begin{remark}
Proposition \ref{Propneg2} completes the proof of Theorem $\ref{aqthe2}$.
\end{remark}

We are now interested in describing negative type endomorphisms. While the existence of positive type endomorphisms did not depend on the Laurent polynomial $a(h)$, the existence of negative type endomorphisms is more subtle and depends both on $a(h)$ and $q$.

\begin{proposition}\label{propneg}
\begin{enumerate}
\item Let $\psi$ is a negative type endomorphism of $A(a(h),q)$. Then $\psi$ acts on the generators of $A(a(h),q)$ as follows:
\begin{equation*}
\psi(h)=\alpha h^{-1},~~\psi(x)= y ch^v, ~~\psi(y)=bh^ux
\end{equation*}
where $\alpha,b,c \in \mathbb{K}^*$ and $u, v\in \mathbb{Z}$ satisfy the relation
\begin{equation}\label{relationa}
bch^{u+v} a(qh) =a(\alpha h^{-1}).
\end{equation}
\item Conversely if  $\alpha,b,c \in \mathbb{K}^*$ and $u, v\in \mathbb{Z}$ satisfy Equation (\ref{relationa}), 
then there exists a unique (negative type) endomorphism $\psi_{(\alpha,b,c,u,v)}$ of $A(a(h),q)$ defined by
\begin{equation*}
\psi_{(\alpha,b,c,u,v)}(h)=\alpha h^{-1},~~\psi_{(\alpha,b,c,u,v)}(x)= y ch^v, ~~\psi_{(\alpha,b,c,u,v)}(y)=bh^ux.
\end{equation*}
\end{enumerate}
\end{proposition}

\begin{proof}
The proof is similar to that of Proposition \ref{propclass1}, however, we include it for the convenience of the reader. 
\begin{enumerate}
\item By assumption $\psi(h)=\alpha h^j$ for $\alpha \in \mathbb{K}^*$ and $j\in \mathbb{Z}_{<0}$. Applying $\psi$ to the relation $xh=qhx$ we get
\begin{equation}\label{eqpq1}
\psi(x)\alpha h^j=q\alpha h^j \psi(x).
\end{equation}
Using the $\mathbb{Z}$-graded structure of $A(a(h),q)$ (see Equation (\ref{grading})), we can write 
$$\psi(x)=\sum_{k\in \mathbb{Z}} w_k,$$
where $w_k\in A_k$ is the component of degree $k$ of $\psi(x)$.  Substituting for $\psi(x)$ in Equation (\ref{eqpq1}), we get
\begin{equation*}
h^{-j}\displaystyle\sum_{k\in \mathbb{Z}} w_k h^j=q \displaystyle\sum_{k\in \mathbb{Z}} w_k.
\end{equation*}
This implies that $q^{jk} w_k=q w_k$ for all $k\in \mathbb{Z}$. By applying $\psi$ to the relation $yx=a$ and comparing degrees in $h$ we see $\psi(x)$ is non-zero. Therefore there is at least one $k$ such that $w_k\neq 0$ and since $q$ is not a root of unity and $j<0$, $q^{jk}=q$ implies that $j=-1$ and $k=-1$ (for all $k$ such that $w_k \neq 0$). Repeating this argument on the relation $yh=q^{-1}hy$ gives us the following action of $\psi$ on the generators of $A(a(h),q)$:
\begin{equation*}
\psi(h)=\alpha h^{-1}, ~~\psi(x)=y c(h), ~~ \psi(y)=b(h) x,
\end{equation*}
for Laurent polynomials $b(h)$ and $c(h)$. Applying $\psi$ twice to $x$ gives us $\psi^2(x)=b(h)c(q^{-1}\alpha h^{-1})x$. Since $\psi^2$ is a positive type endomorphism the coefficient of $x$ must be a Laurent monomial by Proposition \ref{propclass1}. Hence, it follows that $b(h)$ and $c(h)$ are Laurent monomials say $bh^v$ and $ch^u$ respectively, where $b,c\in \mathbb{K}^*$ and $v,u\in \mathbb{Z}$. Relation (\ref{relationa})  follows from applying $\psi$ to the relation $yx=a(h)$.
\item It suffices to show that $\psi_{(\alpha,b,c,u,v)}$ is consistent on the defining relations of $A(a(h),q)$. The first two relations follow by the standard calculations
\begin{equation*}
\psi_{(\alpha,b,c,u,v)}(x) \psi_{(\alpha,b,c,u,v)}(h)=ych^v \alpha h^{-1}=q \alpha h^{-1} y c h^v=q \psi_{(\alpha,b,c,u,v)}(h) \psi_{(\alpha,b,c,u,v)}(x)
\end{equation*}
and
\begin{equation*}
\psi_{(\alpha,b,c,u,v)}(y) \psi_{(\alpha,b,c,u,v)}(h)=bh^u x \alpha h^{-1}=b h^u \alpha q^{-1} h^{-1} x=q^{-1}\psi_{(\alpha,b,c,u,v)}(h)\psi_{(\alpha,b,c,u,v)}(y).
\end{equation*}
Next we have 
$$\psi_{(\alpha,b,c,u,v)}(y)\psi_{(\alpha,b,c,u,v)}(x) =bch^{u+v}a(qh)=a(\alpha h^{-1})=\psi_{(\alpha,b,c,u,v)}(a),$$ 
as desired. Finally we have
\begin{equation*}
\psi_{(\alpha,b,c,u,v)}(x)\psi_{(\alpha,b,c,u,v)}(y)=h^{v+u}q^{-v-u}bc a(h).
\end{equation*}
So it follows from Equation (\ref{relationa}) that 
\begin{equation*}
\psi_{(\alpha,b,c,u,v)}(x)\psi_{(\alpha,b,c,u,v)}(y)=a(\alpha qh^{-1})=
\psi_{(\alpha,b,c,u,v)}(a(qh)).
\end{equation*}
\end{enumerate}
\end{proof}
\noindent

 Proposition \ref{propneg} shows that there is a bijection between the set 
\begin{displaymath}
J:=\{(\alpha,b,c,u,v)~|~ \alpha,b,c\in \mathbb{K}^*,~u,v \in \mathbb{Z} \mbox{ satisfying Equation }(\ref{relationa}) \}
\end{displaymath} 
and the set $\mathrm{End}_-$ of negative type endomorphisms. This bijection sends $(\alpha,b,c,u,v)\in J$ to $\psi_{(\alpha,b,c,u,v)}$.

In order to better understand the structure of negative type endomorphisms, we now investigate the composite map of two negative type endomorphisms. It is easy to check that the composite map of two negative type endomorphisms is a positive type endomorphism. More precisely, let $(\alpha,b,c,u,v),(\alpha',b',c',u',v') \in J$. Then one can easily check that
\begin{equation}\label{compositionnegative}
\psi_{(\alpha,b,c,u,v)} \circ \psi_{(\alpha',b',c',u',v')}= \psi_{(\alpha' \alpha^{-1}, u-v',bc'\alpha^{v'}q^{-v'})}.
\end{equation}
Note that this implies in particular that $\alpha' \alpha^{-1}$ is a $t^{\mathrm{th}}$ root of unity.

\begin{corollary}\label{corog}
If  $\alpha,b,c \in \mathbb{K}^*$ and $u, v\in \mathbb{Z}$ satisfy Equation (\ref{relationa}), then we also have the relation $bc=b^{-1}c^{-1}\alpha^{-u-v}q^{u+v}$.
\end{corollary}
\begin{proof}
From Proposition \ref{propclass1}  we see that for a positive type endomorphism defined by:
$$\varphi(h)=\alpha' h, ~\varphi(x)=b' h^n x \mbox{ and } \varphi(y)=h^{-n}c'y$$ 
we have the relation $c'=\alpha'^d b'^{-1}$ (this relation is a consequence of Proposition \ref{propclass1}). If we apply this relation to the positive type endomorphism $\psi_{(\alpha,b,c,u,v)} ^2=\psi_{(1,u-v,bc\alpha^v q^{-v})}$, then we obtain the result.
\end{proof}

\noindent We now describe explicitly the inverse of a negative type endomorphism (which is an automorphism by Proposition \ref{Propneg2}).  Recall that $\psi_{(1,0,1)}=\mathrm{id}$. Thus by Corollary \ref{corog} and Equation (\ref{compositionnegative}) we have the following.  
  
\begin{proposition}
If  $(\alpha,b,c,u,v)\in J$, then  $(\alpha,q^vc^{-1}\alpha^{-v},q^ub^{-1}\alpha^{-u},v,u)\in J$ and 
$$\psi_{(\alpha,b,c,u,v)}^{-1}= \psi_{(\alpha,q^vc^{-1}\alpha^{-v},q^ub^{-1}\alpha^{-u},v,u)}.$$
\end{proposition}

\section{Endomorphisms and automorphisms of quantum generalized Weyl algebras}

As every endomorphism of $A=A(a(h),q)$ is of positive, zero or negative type, we deduce from Propositions \ref{propclass1}, \ref{propex}, \ref{Propneg2} and \ref{propneg},  and Corollaries \ref{endaut+} and \ref{endaut0} the following description for the set $\mathrm{End}(A)$ of all endomorphisms of $A$ and $\mathrm{Aut}(A)$ of all automorphisms of $A$. 

\begin{theorem}\label{thm:main}
\begin{align*}
\mathrm{End}(A)  = \{\psi_{(\alpha,n,b)} ~|~ (\alpha,n,b) \in I\} &\biguplus 
 \{\psi_{\alpha} ~|~  \alpha  \in \mathbb{K}^* ~\mathrm{and}~a(\alpha )=a(q \alpha)=0\}   \\
&\biguplus  \{\psi_{(\alpha,b,c,u,v)} ~|~ (\alpha,b,c,u,v) \in J \}. 
 \end{align*}
 and  
 \begin{eqnarray*}
\mathrm{Aut}(A) & = & \{\psi_{(\alpha,n,b)} ~|~ (\alpha,n,b) \in I\} \biguplus  \{\psi_{(\alpha,b,c,u,v)} ~|~ (\alpha,b,c,u,v) \in J  \}
 \end{eqnarray*}
where
\begin{displaymath}
 I:=\{(\alpha,n,b)| \alpha \in U_{t}, n\in \mathbb{Z}, b\in \mathbb{K}^*\}
\end{displaymath} 
 and 
\begin{displaymath}
J:=\{(\alpha,b,c,u,v)~|~ \alpha,b,c\in \mathbb{K}^*,~u,v \in \mathbb{Z}~~\mathrm{satisfying ~ Equation~ (\ref{relationa})} \}.
\end{displaymath} 
\end{theorem}

Theorem \ref{aqthe2} is a consequence of Theorem \ref{thm:main}, and Theorem \ref{aqthe} is an easy consequence of Theorem \ref{aqthe2} as explained in the introduction.

Finally, let us denote by $\mathrm{End}_-$ the set of negative type endomorphisms. Assume that $\mathrm{End}_-$ is not empty. That is we assume that there exist  $\alpha,b,c \in \mathbb{K}^*$ and $u, v\in \mathbb{Z}$ satisfying Equation (\ref{relationa}). Fix a negative endomorphism, say $\theta$. Then 
$$ \mathrm{End}_-=\{ \psi \theta ~|~ \psi \in \mathrm{End}_+\} = \{ \theta \psi ~|~ \psi \in \mathrm{End}_+\}.$$
We are now ready to describe the automorphism group of the quantum generalized Weyl algebra $A(a(h),q)$.

\begin{theorem}\label{propaut}
Let $A=A(a(h),q)$ be a quantum generalized Weyl algebra. Recall that $\mathrm{End}_+= \{\psi_{(\alpha,n,b)} ~|~ (\alpha,n,b) \in I\}$.
\begin{enumerate}
\item If $A$ has no negative type endomorphisms, then $\mathrm{Aut(A)}=\mathrm{End}_+$.
\item Assume that $A$ has at least one negative endomorphism, say $\theta$. Then we have the following short exact sequence
\begin{equation*}
1\longrightarrow   \mathrm{End}_+     \longrightarrow \mathrm{Aut(A)}  \longrightarrow \mathbb{Z}/2 \longrightarrow 1.
\end{equation*}
\end{enumerate}
\end{theorem}

We note that the existence of an involution for $A(a(h),q)$ heavily depends on $a(h)$. For instance, when $d+e$ is odd, one can easily show that $A(a(h),q)$ has no involutions. In particular, the algebras $A_{\alpha,q}$ defined in the Introduction have no involutions, but have many negative type endomorphisms. On the other hand, the minimal primitive factors of the quantized enveloping algebra of $U_q(\mathfrak{sl}_2)$ always have involutions, therefore for these algebras the short exact sequence in Theorem \ref{propaut} splits.

\ \\ \hspace{1cm}
\begin{minipage}[c]{\linewidth}
~\\
\noindent 
A.P. Kitchin and S. Launois\\ 
School of Mathematics, Statistics \& Actuarial Science,\\ 
University of Kent, \\
Canterbury, Kent CT2 7NF, United Kingdom\\
E-mails: {\tt ak562@kent.ac.uk} and {\tt S.Launois@kent.ac.uk}

\end{minipage}

\end{document}